\documentclass{amsart}
\usepackage{amsmath, amssymb}
\usepackage[colorlinks=true]{hyperref}

\theoremstyle{plain}
 \newtheorem{theorem}{Theorem}[section]
 \newtheorem{lemma}{Lemma}[section]
 \newtheorem{proposition}{Proposition}[section]

\theoremstyle{definition}

\DeclareMathOperator{\Impart}{Im}

\newcommand{\Z}{\mathbb{Z}}
\newcommand{\Q}{\mathbb{Q}}
\newcommand{\C}{\mathbb{C}}

\title[An expl.\ Andr\'e-Oort type result for $\mathbb{P}^1(\mathbb{C}) \times \mathbb{G}_m(\mathbb{C})$ based on log.\ forms]{An explicit Andr\'e-Oort type result for $\mathbb{P}^1(\mathbb{C}) \times \mathbb{G}_m(\mathbb{C})$ based on logarithmic forms}

\author{Roland Paulin}
\address{Roland Paulin, Department of Mathematics, University of Salzburg, Hellbrunnerstr.\ 34/I, 5020 Salzburg, Austria}
\email{paulinroland@gmail.com}
\thanks{The author was supported by the Austrian Science Fund (FWF): P24574.}

\date{\today}

\begin{document}

\begin{abstract}
Using linear forms in logarithms we prove an explicit result of Andr\'e-Oort type for $\mathbb{P}^1(\C) \times \mathbb{G}_m(\C)$.
In this variation the special points of $\mathbb{P}^1(\C) \times \mathbb{G}_m(\C)$ are of the form $(\alpha, \lambda)$, with $\alpha$ a singular modulus and $\lambda$ a root of unity.
The qualitative version of our result states that if $\mathcal{C}$ is a closed algebraic curve in $\mathbb{P}^1(\C) \times \mathbb{G}_m(\C)$, defined over a number field, not containing a horizontal or vertical line, then $\mathcal{C}$ contains only finitely many special points.
The proof is based on linear forms in logarithms.
This differs completely from the method used by the author recently in the proof of the same kind of statement, where class field theory was applied.
\end{abstract}

\maketitle

\section{Introduction and result} \label{sec:Intro}
K\"uhne in \cite{Kuehne:An_effective_result_of_Andre-Oort_type} and independently Bilu, Masser and Zannier in \cite{Bilu-Masser-Zannier} have studied the Andr\'e-Oort conjecture in the case of the Shimura variety $\mathbb{P}^1(\C) \times \mathbb{P}^1(\C)$, where $\mathbb{P}^1(\C)$ is the modular curve $\operatorname{SL}_2(\Z) \backslash \mathcal{H}^*$.
They obtain the first nontrivial, unconditional, effective results in the area.
In \cite{Paulin:Andre-Oort-P1xGm} the author investigates the variant $\mathbb{P}^1(\C) \times \mathbb{G}_m(\C)$, where the special points are of the form $(j(\tau), \lambda)$, with $\tau$ an imaginary quadratic number and $\lambda \in \C$ a root of unity.

In this article we attack the same problem as in \cite{Paulin:Andre-Oort-P1xGm}, using a different method.
We prove a weaker version of the main explicit result of \cite{Paulin:Andre-Oort-P1xGm}, therefore we also reprove the main non-effective result.
The better bounds in \cite{Paulin:Andre-Oort-P1xGm} are achieved using more sophisticated class field theory.
We use less class field theory here, and instead are forced to apply linear forms in logarithms.
Even though the bounds are worse here, the methods presented could still be useful for other similar problems.

Let $\mathcal{H}$ denote the complex upper half-plane.
We call $(\alpha, \lambda) \in \mathbb{P}^1(\C) \times \mathbb{G}_m(\C)$ a special point, if $\alpha = j(\tau)$ for some imaginary quadratic $\tau \in \mathcal{H}$ and $\lambda \in \C$ is a root of unity.
We work with the same assumptions as in Theorem 2 of \cite{Paulin:Andre-Oort-P1xGm}.
So $K$ is a number field of degree $d$ over $\Q$ with a fixed embedding into $\C$, and $F \in K[X,Y]$ is a nonconstant polynomial with $\delta_1 = \deg_X F$ and $\delta_2 = \deg_Y F$.
We assume that zero set of $F(X,Y) = 0$ contains no vertical or horizontal line, i.e.\ $F(X,Y)$ does not have a nonconstant divisor $f \in K[X]$ or $g \in K[Y]$.
Then clearly $\delta_1, \delta_2 > 0$.
Let $h(F)$ denote the height of the polynomial $F$ (so $h(F)$ is the absolute logarithmic Weil height of the point defined by the nonzero coefficients of $F$ in projective space, see the definition in section \ref{sec:Prelim}).
Let $(\alpha, \lambda)$ be a special point of $\mathcal{C}$, where $\alpha = j(\tau)$ for some $\tau \in \mathcal{H}$.
Let $\Delta$ denote the discriminant of the endomorphism ring of the complex elliptic curve $\C/(\Z + \Z \tau)$, and let $N$ be the smallest positive integer such that $\lambda^N = 1$.

\begin{theorem} \label{thm:main}
In the above situation
\begin{equation} \label{eq:main-Q-Delta}
|\Delta| < \left(\frac{1}{d \delta_2} C \log C\right)^2
\end{equation}
and
\begin{equation} \label{eq:main-Q-N}
N < C (\log C)^2 \log \log C
\end{equation}
with
\[
C = 2^{36} d^3 \delta_2^3 (\log(4 d \delta_2))^2 \max \left(d h(F) + (d-1) (\delta_1 + \delta_2) \log 2, 1\right).
\]
\end{theorem}
This result implies Theorem 1 of \cite{Paulin:Andre-Oort-P1xGm}, because there are only finitely many special points $(\alpha, \lambda)$ with $\Delta$ and $N$ bounded.
Similarly to \cite{Paulin:Andre-Oort-P1xGm}, we can reduce the proof to the case when $K = \Q$ and $\Z + \Z\tau$ is an order.

In the following section we collect some preliminary definitions and statements, and also prove some auxiliary results.
In section \ref{sec:proof} we prove Theorem \ref{thm:main}.

\section{Preliminaries} \label{sec:Prelim}

For the readers convenience we recall some definitions.
The (absolute logarithmic Weil) height of a point $P = (a_0: \dotsc: a_n) \in \mathbb{P}^n_{\overline{\Q}}$ is defined by
\[
h(P) = \sum_{v \in M_K} \frac{[K_v:\Q_v]}{[K:\Q]} \log(\max_i |a_i|_v),
\]
where $K$ is any number field containing all $a_i$, $M_K$ is the set of places of $K$, and for any place $v$, $|\cdot|_v$ is the absolute value on $K$ extending a standard absolute value of $\Q$.
Similarly, the (absolute logarithmic Weil) height of a polynomial $F \in \overline{\Q}[X_1, \dotsc, X_n]$ with nonzero coefficients $c_i$ is defined by
\[
h(F) = \sum_{v \in M_K} \frac{[K_v:\Q_v]}{[K:\Q]} \log(\max_i |c_i|_v),
\]
where $K$ is a number field containing the coefficients of $F$.
We use the notation $H(F) = e^{h(F)}$.
If $F \in \Z[X_1, \dotsc, X_n]$, and the gcd of the coefficients of $F$ is $1$, then $H(F)$ is equal to the maximum of the euclidean absolute values of the coefficients of $F$.

If $K$ is a number field and $\alpha \in K$, then the (absolute logarithmic Weil) height of $\alpha$ is
\[
h(\alpha) = h(\alpha:1) = \sum_{v \in M_K} \frac{[K_v:\Q_v]}{[K:\Q]} \log \max(1, |\alpha|_v).
\]
We use the notation $H(\alpha) = e^{h(\alpha)}$.

Let $f = a_d X^d + \dotsm + a_0 = a_d (X-\alpha_1) \dotsm (X-\alpha_d) \in \C[X]$, where $a_d \neq 0$.
The Mahler measure of $f$ (see e.g.\ \cite{Bombieri-Gubler}) is defined by
\[
M(f) = \exp\left(\int_0^1 \log|f(e^{2\pi i t})|dt\right) = |a_d| \prod_{j=1}^d \max(1, |\alpha_j|).
\]
If $\alpha \in \overline{\Q}$ and $f \in \Z[X]$ is the minimal polynomial of $\alpha$, then $h(\alpha) = \frac{\log M(f)}{\deg f}$ (see Proposition 1.6.6 in \cite{Bombieri-Gubler}).

If $\mathcal{O}$ is an order in an imaginary quadratic number field $L$, then the class number $h(\mathcal{O})$ denotes the number of equivalence classes of proper fractional ideals of $\mathcal{O}$ (see e.g.\ \cite{Cox:Primes_of_the_form_x2+ny2}).
Since $\mathcal{O}$ is an order in $L$, we can write it in the form $\Z + \Z \tau_0$ for some $\tau_0$ in $L \cap \mathcal{H}$.
Then the discriminant of the order $\mathcal{O}$ is $D(\mathcal{O}) = -4 (\Impart \tau_0)^2$ (see e.g.\ \S 7, Ch.\ 2 in \cite{Cox:Primes_of_the_form_x2+ny2}).
This is a negative integer congruent to $0$ or $1$ modulo $4$.

If $D<0$ is an integer congruent to $0$ or $1$ modulo $4$, then the class number $h(D)$ denotes the number of proper equivalence classes of primitive quadratic forms of discriminant $D$ (see \cite{Cox:Primes_of_the_form_x2+ny2} or \cite{Hua:Intro_to_Number_Theory}).
Theorem 7.7 in \cite{Cox:Primes_of_the_form_x2+ny2} says that if $\mathcal{O}$ is an order with discriminant $D$ in an imaginary quadratic number field, then $h(\mathcal{O}) = h(D)$.

The following theorem is the main result of \cite{BaW:Log_forms_and_group_varieties}.
\begin{theorem}[Baker, W\"ustholz] \label{thm:BaWu}
Let $\alpha_1, \dotsc, \alpha_n \in \C \setminus \{0,1\}$ be algebraic numbers, with fixed determinations of logarithms $\log \alpha_1, \dotsc, \log \alpha_n$.
The degree of the field extension $\Q(\alpha_1, \dotsc, \alpha_n)/\Q$ is denoted by $d$.
Let $L(z_1, \dotsc, z_n) = b_1 z_1 + \dotsm + b_n z_n$ be a linear form, where $b_1, \dotsc, b_n$ are integers such that at least one $b_i$ is nonzero.
We use the notation $h'(\alpha_i) = \max(h(\alpha_i), \frac{1}{d} |\log \alpha_i|, \frac{1}{d})$ (which depends on the choice of $\log \alpha_i$) and $h'(L) = \max(h(L),\frac{1}{d})$, where $h(L)$ is the absolute logarithmic Weil height of $(b_1:\dotsc:b_n)$ in $\mathbb{P}^{n-1}_{\Q}$.
If $\Lambda = L(\log \alpha_1, \dotsc, \log \alpha_n) \neq 0$, then
\[
\log |\Lambda| > -C(n,d) h'(\alpha_1) \dotsm h'(\alpha_n) h'(L),
\]
where
\[
C(n,d) = 18 (n+1)! n^{n+1} (32d)^{n+2} \log(2nd).
\] 
\end{theorem}

The following lemma gives us a lower bound on the distance between an algebraic number and a root of unity.
\begin{lemma} \label{lemma:distance-root-of-unity}
Let $\lambda \in \C$ be an $N^{\textrm{th}}$ root of unity, where $N$ is a positive integer.
If $\gamma \in \C$ is algebraic of degree $d$ over $\Q$, and $\lambda \neq \gamma$, then
\[
\log |\lambda - \gamma| > -c d^3 \log(4d) \max\left(h(\gamma), \frac{4}{d} \right) \log \max(N,2)
\]
with $c = 2^{25} 3^3 \pi + 1$.
\end{lemma}
\begin{proof}
The right hand side is the same for $N=1$ and for $N=2$, so we may assume that $N \ge 2$.
Suppose first that $d=1$.
Then $\gamma \in \Q$, so there are coprime integers $a,b$ such that $b > 0$ and $\gamma = \frac{a}{b}$.
If $\lambda \in \{1, -1\}$, then $|\lambda - \gamma| = \left|\frac{a \pm b}{b}\right| \ge \frac{1}{b} \ge \frac{1}{H(\gamma)}$, hence $\log |\lambda-\gamma| \ge -h(\gamma)$.
Now let $\lambda \notin \{1, -1\}$, then $N \ge 3$, and $|\lambda - \gamma| \ge |\Impart(\lambda)| \ge \sin(\frac{2 \pi}{N})$.
The sine function is concave in the interval $[0, \pi]$, so $\sin x \ge \frac{\sin(2\pi/3)}{2\pi/3} x = \frac{3\sqrt{3}}{4 \pi} x$ for every $x \in [0, \frac{2\pi}{3}]$.
Thus $|\lambda - \gamma| \ge \frac{3\sqrt{3}}{4 \pi} \cdot \frac{2 \pi}{N} = \frac{3\sqrt{3}}{2N} \ge \frac{1}{N}$, therefore $\log |\lambda - \gamma| \ge - \log N$.
So the statement is true for $d=1$.

From now on we assume that $N, d \ge 2$.
If $|\lambda-\gamma| \ge \frac{1}{4}$, then $\log|\lambda-\gamma| \ge -\log 4$, which is clearly greater than the bound needed.
So we may assume that $|\lambda-\gamma|<\frac{1}{4}$.
Let us define the logarithm function in the unit disc around $1$ such that $\log(1) = 0$, and let $s = \log(\frac{\gamma}{\lambda})$.
Here $s$ is well defined, because $|\frac{\gamma}{\lambda}-1| = |\gamma - \lambda| < \frac{1}{2}$.
It is a basic fact from analysis that if $z \in \C$ and $|z|<\frac{1}{2}$, then $|\log(1+z)| \le 2 |z|$.
Using this for $z = \frac{\gamma}{\lambda}-1$, we get that $|s| \le 2|\frac{\gamma}{\lambda}-1| = 2|\lambda-\gamma|$.
In particular $|s| < \frac{1}{2}$.

We can find an integer $u$ such that $|u| \le \frac{N}{2}$ and $\lambda = e^{\frac{u}{N} 2\pi i}$.
Note that $e^{s + \frac{u}{N} 2\pi i} = \gamma$ and $e^{\pi i} = -1$, so we may choose the logarithms $\log \gamma$ and $\log(-1)$ to be $s + \frac{u}{N} 2\pi i$ and $\pi i$.
Define the linear form $L(z_1,z_2) = N z_1 - 2u z_2$.
We will apply the Baker-W\"ustholz estimate (Theorem \ref{thm:BaWu}) for $\Lambda = L(\log \gamma, \log(-1))$.
Here $\Lambda = N \log \gamma - 2u \pi i = N(s+\frac{u}{N} 2\pi i) - 2u \pi i = N s \neq 0$, because otherwise $\log \gamma = \frac{u}{N} 2\pi i$, so $\gamma = e^{\frac{u}{N} 2\pi i} = \lambda$, contradicting $\lambda \neq \gamma$.
Note that $-1, \gamma \notin \{0,1\}$, because $d \ge 2$.
Theorem \ref{thm:BaWu} tells us that
\[
\log |\Lambda| > -C(2,d) h'(\gamma) h'(-1) h'(L),
\]
where 
\begin{align*}
C(2,d) &= 18 \cdot 6 \cdot 8 (32d)^4 \log(4d), \\
h'(\gamma) &= \max\left(h(\gamma), \frac{1}{d} |\log \gamma|, \frac{1}{d}\right), \\
h'(-1) &= \max\left(h(-1), \frac{1}{d}|\log(-1)|, \frac{1}{d}\right) = \frac{\pi}{d}, \\
h'(L) &= \max\left(h\left(\frac{2u}{N}\right), \frac{1}{d}\right) \le \max\left(\log N, \frac{1}{d}\right) \le \max\left(\log N, \frac{1}{2}\right) = \log N.
\end{align*}
Note that $|\log \gamma| = |s+\frac{2u}{N} \pi i| \le |s| + \pi \le \pi + \frac{1}{2} < 4$, so $h'(\gamma) \le \max(h(\gamma), \frac{4}{d})$.
Collecting these inequalities together, we get
\[
\log |\Lambda| > - c' d^3 \log(4d) \max\left(h(\gamma), \frac{4}{d}\right) \log N,
\]
where $c' = 18 \cdot 6 \cdot 8 \cdot 32^4 \cdot \pi = 2^{25} 3^3 \pi = c-1$.
From $|s| \le 2|\lambda-\gamma|$ and $\Lambda = Ns$ we obtain that $|\lambda-\gamma| \ge \frac{|\Lambda|}{2N}$.
So
\begin{align*}
\log |\lambda-\gamma| &\ge \log|\Lambda| - \log(2N) \ge \log|\Lambda| - 2 \log N \\
&> -(c'+1) d^3 \log(4d) \max\left(h(\gamma), \frac{4}{d}\right) \log N \\
&= - c d^3 \log(4d) \max\left(h(\gamma), \frac{4}{d}\right) \log N.
\end{align*}
\end{proof}

In the following lemma, we get a bound for the value of a polynomial at a root of unity.
\begin{lemma} \label{lemma:lower-bound-poly-root-of-unity}
Let $N$ and $d$ be positive integers.
Let $\lambda \in \C$ be an $N^{\textrm{th}}$ root of unity, and let $g \in \Z[X]$ be a polynomial of degree at most $\delta$ such that $g(\lambda) \neq 0$.
Then
\[
\log |g(\lambda)| > - 2^{35} \delta^2 (\log(4\delta))^2 \max(h(g),1) \log \max(N,2).
\]
\end{lemma}
\begin{proof}
We will prove the inequality
\begin{equation} \label{eq:g(lambda)-Mahler-bound}
\log |g(\lambda)| > - c_2 \delta^2 \log(4\delta) \log(2 M(g)) \log \max(N,2),
\end{equation}
where $c_0 = 2^{25} 3^3 \pi + 1$, $c_1 = \frac{4}{\log 2} c_0$ and $c_2 = c_1 + 6$.

We argue by induction on $\delta$.
The right hand side of \eqref{eq:g(lambda)-Mahler-bound} is the same for $N=1$ and for $N=2$, so we may assume that $N \ge 2$.
If $\deg g = 0$, then $g(\lambda) \in \Z \setminus \{0\}$, so $\log |g(\lambda)| \ge 0$.
Now let $\deg g \ge 1$.
The right hand side of \eqref{eq:g(lambda)-Mahler-bound} is a monotone decreasing function of $\delta$, so we may assume that $\deg g = \delta$, and that \eqref{eq:g(lambda)-Mahler-bound} is true for smaller values of $\delta$.
We may also assume that the gcd of the coefficients of $g$ is $1$, because multiplying $g$ by a positive integer increases the left hand side and decreases the right hand side of \eqref{eq:g(lambda)-Mahler-bound}.
Suppose $g$ is not irreducible, then $g = g_1 g_2$ for some polynomials $g_1, g_2 \in \Z[X]$ of positive degrees $d_1$ and $d_2$.
Note that $\delta = d_1 + d_2$ and $M(g) = M(g_1) M(g_2) \ge M(g_1), M(g_2)$.
Using the induction hypothesis for $g_1$ and $g_2$, we get
\begin{align*}
\log |g(\lambda)| &= \log |g_1(\lambda)| + \log |g_2(\lambda)| \\
&\ge - c_2 \log(2 M(g)) (\log N) \left( d_1^2 \log(4d_1) + d_2^2 \log(4d_2) \right) \\
&\ge - c_2 \log(2 M(g)) (\log N) \delta^2 \log(4\delta),
\end{align*}
because
\[
\delta^2 \log(4\delta) = (d_1^2 + d_2^2 + 2d_1 d_2) \log(4\delta) \ge d_1^2 \log(4d_1) + d_2^2 \log(4d_2).
\]

Finally, let $g$ be irreducible.
Let $g = a (X-\gamma_1) \dotsm (X-\gamma_{\delta})$, where $a \in \Z \setminus \{0\}$ and $\gamma_1, \dotsc, \gamma_{\delta} \in \C$.
Choose a $k \in \{1, \dotsc, n\}$ such that $|\lambda - \gamma_k|$ is minimal.
During the proof of Theorem A.3 in \cite{Bugeaud} it is shown that if $P \in \Z[X] \setminus \{0\}$ is separable polynomial of degree $n$, and $\alpha, \beta \in \C$ are such that $P(\alpha) = P(\beta) = 0$ and $\alpha \neq \beta$, then
\[
|\alpha - \beta|^2 \frac{n^3}{3} \max(1,|\alpha|, |\beta|)^{-2} n^{n-1} M(P)^{2n-2} > 1.
\]
So in fact
\[
|\alpha-\beta| > \sqrt{3} n^{-(n+2)/2} \max(1, |\alpha|, |\beta|) M(P)^{-(n-1)} \ge \sqrt{3} n^{-(n+2)/2} M(P)^{-(n-1)}.
\]
Applying this result for $P=g$ and $\alpha = \gamma_k$, $\beta = \gamma_i$ with $i \neq k$, we get that $|\gamma_k - \gamma_i| > R$ with $R = \sqrt{3} \delta^{-(\delta+2)/2} M(g)^{-(\delta-1)}$.
Then
\[
R < |\gamma_k - \gamma_i| \le |\gamma_k - \lambda| + |\lambda - \gamma_i| \le 2 |\lambda - \gamma_i|,
\]
so
\begin{equation} \label{eq:lambda-gammai-bound}
|\lambda - \gamma_i| > \frac{R}{2}
\end{equation}
for every $i \neq k$.
Applying Lemma \ref{lemma:distance-root-of-unity} and using $\log M(g) = \delta h(\gamma_k) \ge 0$, we obtain
\begin{equation} \label{eq:lambda-gammak-bound}
\begin{split}
\log |\lambda - \gamma_k| &> -c_0 \delta^2 \log(4\delta) \max\left(\log M(g), 4 \right) \log N \\
&\ge -c_1 \delta^2 \log(4\delta) \log(2M(g)) \log N.
\end{split}
\end{equation}
Let $A = \delta^2 \log(4\delta) \log(2M(g)) \log N$.
The bounds \eqref{eq:lambda-gammai-bound} and \eqref{eq:lambda-gammak-bound} together imply that
\begin{align*}
\log |g(\gamma)| &> (\delta-1) \log\left(\frac{R}{2}\right) - c_1 A \\
&= -(\delta-1) \log\left(\frac{2}{\sqrt{3}}\right) - \frac{(\delta-1)(\delta+2)}{2} \log \delta - (\delta-1)^2 \log M(g) - c_1 A.
\end{align*}
It is easy to check that the terms
\[
(\delta-1) \log\left(\frac{2}{\sqrt{3}}\right), \quad \frac{(\delta-1)(\delta+2)}{2} \log \delta, \quad (\delta-1)^2 \log M(g)
\]
are all smaller than $2A$, so $\log |g(\gamma)| > -(c_1 + 6) A = - c_2 A$.
This finishes the proof of \eqref{eq:g(lambda)-Mahler-bound}.

To prove the statement of the lemma, first note that we may assume that the gcd of the coefficients of $g$ is $1$.
Then every coefficient of $g$ has euclidean absolute value at most $H(g)$, so $M(g) \le \sqrt{\delta+1} H(g)$ (see Lemma 1.6.7 in \cite{Bombieri-Gubler}), hence
\begin{align*}
\log |g(\lambda)| &> - c_2 \delta^2 \log(4\delta) \log(2 \sqrt{\delta+1} H(g)) \log \max(N,2) \\
&\ge - 2 c_2 \delta^2 \log(4\delta)^2 \max(h(g),1) \log \max(N,2),
\end{align*}
because
\[
\log 2 + \frac{\log(\delta+1)}{2} + h(g) \le 2 \log(4\delta) \max(h(g),1).
\]
The statement of the lemma follows from $2 c_2 < 2^{35}$.
\end{proof}

If $\lambda$ is a primitive $N^{\textrm{th}}$ root of unity, then the degree of $\lambda$ over $\Q$ is $\varphi(N)$, where $\varphi$ denotes Euler's totient function.
During the proof of Theorem \ref{thm:main} we will need a lower bound for the degree of $\lambda$.
A more or less trivial bound would be $\varphi(N) \ge c(\varepsilon) N^{1-\varepsilon}$ for every positive $\varepsilon$.
We can actually do better.
\begin{proposition} \label{prop:lower-bound-phi}
If $N > 30$ is an integer, then $\varphi(N) > \frac{N}{3 \log \log N}$.
\end{proposition}
\begin{proof}
The statement can be easily verified case by case for $31 \le N \le 66$, so we may assume that $N \ge 67$.
Theorem 15 in \cite{Rosser-Schoenfeld:Approx_formulas} implies that
\[
\varphi(N) > \frac{N}{e^{\gamma} \log \log N + \frac{2.51}{\log \log N}}
\]
for $N \ge 3$, with $\gamma$ denoting the Euler constant.
Now $\log \log N \ge \log \log 67 > \sqrt{\frac{2.51}{3-e^{\gamma}}}$, hence $(3-e^{\gamma}) \log \log N > \frac{2.51}{\log \log N}$, therefore
\[
\varphi(N) > \frac{N}{e^{\gamma} \log \log N + \frac{2.51}{\log \log N}} > \frac{N}{3 \log \log N}.
\]
\end{proof}

We will use the following upper bound for the class number $h(D)$.
\begin{proposition} \label{prop:class-number-bound}
If $D<0$ is an integer congruent to $0$ or $1$ modulo $4$, then
\[
h(D) < \frac{1}{\pi} \sqrt{|D|} (2 + \log|D|).
\]
\end{proposition}
\begin{proof}
The class number formula (see Theorem 10.1, Ch.\ 12 in \cite{Hua:Intro_to_Number_Theory}) says that
\[
h(D) = \frac{w \sqrt{|D|}}{2\pi} L_D(1),\]
where $w$ is equal to $6$, $4$ and $2$ for $D=-3$, $D=-4$ and $D<-4$ respectively, and $L_D(1) = \sum_{n=1}^{\infty} \frac{1}{n} (\frac{D}{n})$.
If $D=-3$ or $-4$, then $h(D)=1$, thus the statement of the proposition is true.
So we may assume that $D<-4$.
Then $w=2$, so $h(D) = \frac{\sqrt{|D|}}{\pi} L_D(1)$.
Theorem 14.3, Ch.\ 12 in \cite{Hua:Intro_to_Number_Theory} says that $0 < L_D(1) < 2 + \log |D|$, which gives $h(D) < \frac{\sqrt{|D|}}{\pi}(2 + \log |D|)$.
\end{proof}

We will use the following auxiliary lemma in the proof of Theorem \ref{thm:main}.
\begin{lemma} \label{lemma:auxiliary}
Let $p > 0$, $q > e$ and $A > 10^4$ be real numbers such that
\[
p \le A \log q \qquad \textrm{and} \qquad \frac{q}{\log\log q} \le p \log p.
\]
Then $p < 2 A \log A$ and $q < 3 A (\log A)^2 \log \log A$.
\end{lemma}
\begin{proof}
Note that $0 < \frac{q}{\log \log q} \le p \log p$, so $p > 1$.
Since $p \log p$ is an increasing function of $p \in (1, \infty)$, we may assume that $p = A \log q$.
Then $\frac{q}{\log \log q} \le A (\log q) \log(A \log q)$, so $q \le G(q)$, where
\[
G(x) = A (\log x) (\log \log x) \log(A \log x).
\]
We claim that $G(x)/x$ is strictly decreasing function of $x$ in the interval $(e^4, \infty)$.
Indeed, if $x > e^4$, then $\log(Ax) \ge \log \log x > 1$, so
\[
G'(x) = \frac{A}{x} ((1+\log \log x)(1+\log(A\log x))-1) < \frac{4A}{x} (\log \log x) \log(A \log x),
\]
hence
\[
(G(x)/x)' = \frac{1}{x^2}(G'(x)x-G(x)) < \frac{A}{x^2} (4 - \log x) (\log \log x) \log(A \log x) < 0.
\]

We claim that $G(Q) < Q$, where $Q = 3 A (\log A)^2 \log \log A$.
This will prove the upper bound on $q$, because $G(Q)/Q < 1 \le G(q)/q$ implies $q < Q$, since $G(x)/x$ is a decreasing function in $(e^4, \infty)$, and $Q > e^4$.
So we need to show
\[
A (\log Q) (\log \log Q) \log(A \log Q) < 3 A (\log A)^2 \log \log A,
\]
or equivalently
\[
\frac{\log Q}{\log A} \cdot \frac{\log(A \log Q)}{\log A} \cdot \frac{\log \log Q}{\log \log A} < 3.
\]
One can easily check that $3 < A^{1/8}$, $\log A < A^{1/4}$ and $\log \log A < A^{1/8}$ for $A > 10^4$.
Thus $3 (\log A)^2 (\log \log A) < A^{3/4}$, hence $Q < A^{7/4}$ and $\frac{\log Q}{\log A} < \frac{7}{4}$.
Then
\[
\log Q < \frac{7}{4} \log A < A^{1/3},
\]
because $\frac{7}{4} < A^{1/12}$ and $\log A < A^{1/4}$ for $A>10^4$.
So $A \log Q < A^{4/3}$ and $\frac{\log(A \log Q)}{\log A} < \frac{4}{3}$.
We have
\[
\log \log Q < \log\left(\frac{7}{4} \log A\right) < \log\left((\log A)^{9/7}\right),
\]
because $\frac{7}{4} < (\log A)^{2/7}$ for $A>10^4$.
So $\frac{\log \log Q}{\log \log A} < \frac{9}{7}$.
This proves $q < Q$, because $\frac{7}{4} \cdot \frac{4}{3} \cdot \frac{9}{7} = 3$.

Finally, we have seen that $q < Q < A^{7/4} < A^2$, so $p = A \log q < 2 A \log A$.
\end{proof}

\section{Proof of Theorem \ref{thm:main}} \label{sec:proof}

One can show in the same way as in the proof of Theorem 2 in \cite{Paulin:Andre-Oort-P1xGm}, that it is enough to prove the statement in the case when $K = \Q$ and $\Z+\Z \tau$ is an order.
So let $K = \Q$ and $\mathcal{O} = \Z + \Z \tau$ an order.
Then $d=1$ and
\[
C = 2^{36} \delta_2^3 (\log(4 \delta_2))^2 \max(h(F),1).
\]

The polynomial $g(Y) = F(\alpha,Y) \in \Q(\alpha)[Y]$ is nonzero, because the zero set of $F$ contains no vertical line.
Moreover $g(\lambda) = 0$, so $[\Q(\alpha,\lambda):\Q(\alpha)] \le \deg g \le \delta_2$.
This gives
\[
\varphi(N) = [\Q(\lambda):\Q] \le [\Q(\alpha,\lambda):\Q] = [\Q(\alpha,\lambda):\Q(\alpha)] \cdot [\Q(\alpha):\Q] \le \delta_2 \cdot h(\mathcal{O}).
\]
The discriminant of the order $\mathcal{O}$ is $\Delta = -4 (\Impart \tau)^2$.
We know that $\Delta<0$ and $\Delta$ is congruent to $0$ or $1$ modulo $4$.
Moreover $h(\mathcal{O}) = h(\Delta)$ (see e.g.\ Theorem 7.7 in \cite{Cox:Primes_of_the_form_x2+ny2}).
Using Proposition \ref{prop:class-number-bound}, we obtain that
\begin{equation} \label{eq:phi(N)-bound}
\varphi(N) \le \delta_2 h(\Delta) \le \frac{\delta_2}{\pi} \sqrt{|\Delta|}(2 + \log|\Delta|).
\end{equation}
Suppose $|\Delta| < 25$.
Then \eqref{eq:main-Q-Delta} is true.
If $N \le 30$, then \eqref{eq:main-Q-N} is also true.
If $N > 30$, then using Proposition \ref{prop:lower-bound-phi}, we obtain
\[
\sqrt{N} < \frac{N}{3\log\log N} < \varphi(N) < \frac{\delta_2}{\pi} 5 (2 + \log 25) < 9 \delta_2,
\]
which implies $N < 81 \delta_2^2$.
This proves \eqref{eq:main-Q-N} if $|\Delta| < 25$.
From now on we assume that $|\Delta| \ge 25$.

Since $|\Delta| \ge 25$, we have $\Impart \tau = \frac{\sqrt{|\Delta|}}{2} \ge \frac{5}{2} > \frac{1}{2\pi} \log 6912$, and from \cite[Prop.\ 3.1]{Paulin:Andre-Oort-P1xGm}, we deduce as in \cite{Paulin:Andre-Oort-P1xGm} that $\frac{|\alpha|}{e^{2\pi \Impart \tau}} \in [\frac{1}{2},2]$.
Taking logarithms leads to $\log|\alpha| \ge 2\pi \Impart \tau - \log 2 > 6 \Impart \tau$, because $\Impart \tau \ge \frac{5}{2} > \frac{\log 2}{2\pi - 6}$.
Substituting $\Impart \tau = \frac{\sqrt{|\Delta|}}{2}$ we obtain
\begin{equation} \label{eq:Delta-log-alpha-bound}
3 \sqrt{|\Delta|} < \log|\alpha|.
\end{equation}

We multiply $F$ by a nonzero rational number, so that $F$ will have integer coefficients with gcd equal to $1$.
Then the maximum of the euclidean absolute values of the coefficients of $F$ is $H(F) = e^{h(F)}$.
Let $F = \sum_{i=0}^{\delta_1} g_i(Y) X^i$, where $g_i(Y) \in \Z[Y]$.
Here each $g_i$ has degree at most $\delta_2$.
Since $F(X,\lambda) \in \C[X]$ is a nonzero polynomial, $g_i(\lambda) \neq 0$ for some $i$.
Let $m$ be the maximal such $i$.
It is proved in \cite{Paulin:Andre-Oort-P1xGm} that
\[
|g_m(\lambda)| < \frac{(\delta_2+1) H(F)}{|\alpha|-1}.
\]
Since $|\Delta| \ge 25$, inequality \eqref{eq:Delta-log-alpha-bound} implies that $\log|\alpha|>15$, hence $|\alpha|>e^{15}>2$.
Thus
\[
\frac{(\delta_2+1) H(F)}{|\alpha|-1} \le \frac{4\delta_2 H(F)}{|\alpha|},
\]
therefore
\begin{equation} \label{eq:gm-lambda-upper-bound}
\log |g_m(\lambda)| < \log(4\delta_2) + h(F) - \log|\alpha|.
\end{equation}
On the other hand, we can apply Lemma \ref{lemma:lower-bound-poly-root-of-unity} for $g_m$ and $\lambda$.
Since $\deg g_m \le \delta_2$, and the coefficients of $g_m$ have euclidean absolute values at most $H(F) = e^{h(F)}$, we have $h(g_m) \le h(F)$, and Lemma \ref{lemma:lower-bound-poly-root-of-unity} says
\begin{equation} \label{eq:gm-lambda-lower-bound}
\log |g_m(\lambda)| > -2^{35} \delta_2^2 (\log(4\delta_2))^2 \max(h(F),1) \log \max(N,2).
\end{equation}
The inequalities \eqref{eq:Delta-log-alpha-bound}, \eqref{eq:gm-lambda-upper-bound} and \eqref{eq:gm-lambda-lower-bound} together imply
\begin{equation} \label{eq:final1}
\begin{split}
3 \sqrt{|\Delta|} &< 2^{35} \delta_2^2 (\log(4\delta_2))^2 \max(h(F),1) \log \max(N,2) + \log(4\delta_2) + h(F) \\
&< (2^{35}+2) \delta_2^2 (\log(4\delta_2))^2 \max(h(F),1) \log \max(N,2).
\end{split}
\end{equation}
If $N \le 30$, then we get
\[
\sqrt{|\Delta|} < 2^{36} \delta_2^2 (\log(4\delta_2))^2 \max(h(F),1) = \frac{1}{\delta_2} C < \frac{1}{\delta_2} C \log C,
\]
hence both \eqref{eq:main-Q-Delta} and \eqref{eq:main-Q-N} are true.
From now on we assume that $N > 30$.

Applying \eqref{eq:phi(N)-bound} and Proposition \ref{prop:lower-bound-phi}, we get
\begin{equation} \label{eq:final2}
\frac{N}{3\log\log N} < \frac{\delta_2}{\pi} \sqrt{|\Delta|}(2 + \log|\Delta|).
\end{equation}
Let $p = \frac{12}{\pi} \delta_2 \sqrt{|\Delta|}$ and $q = 2 N$, then \eqref{eq:final1} and \eqref{eq:final2} imply
\[
p < A \log q \qquad \textrm{and} \qquad
\frac{q}{\log\log q} < p \log p
\]
with $A = \frac{4}{\pi} (2^{35}+2) \delta_2^3 (\log(4\delta_2))^2 \max(h(F),1)$.
Applying Lemma \ref{lemma:auxiliary}, we obtain
\[
p < 2 A \log A \qquad \textrm{and} \qquad q < 3 A (\log A)^2 \log \log A.
\]
Then $\sqrt{|\Delta|} < \frac{\pi}{6 \delta_2} A \log A$ and $N < \frac{3}{2} A (\log A)^2 \log \log A$.
The inequalities \eqref{eq:main-Q-Delta} and \eqref{eq:main-Q-N} follow from these, because $\frac{\pi}{6} A < A < \frac{3}{2} A < C$.

\subsection*{Acknowledgements}
This paper has its origins in the author's Ph.D.\ studies under the supervision of Gisbert W\"ustholz at ETH Z\"urich.
Therefore the author thanks Gisbert W\"ustholz for introducing him to this field, and for all the helpful discussions.

\bibliographystyle{hplain}
\bibliography{Andre-Oort-P1xGm-log-forms}

\end{document}